\documentclass[12pt]{article}
\usepackage{amssymb,amsmath,amsthm,bm}
\usepackage{graphicx}
\usepackage[longnamesfirst]{natbib}
\usepackage{xcolor} 
\usepackage[hidelinks]{hyperref} 
\newtheorem{theorem}{Theorem}[section]
\newtheorem{assumption}{Assumption}[section]

\newtheorem{remark}{Remark}[section]

\numberwithin{equation}{section}
\numberwithin{example}{section}

\begin{document}

\title{On (in)consistency of {M}-estimators under contamination\footnote{We thank Mikhail Zhelonkin and Chen Zhou for helpful comments.}}

\author{Jens Klooster\footnote{Department of Economics, Econometrics and Finance, University of Groningen, The Netherlands. E-mail: \url{j.klooster@rug.nl}}\ \ \& Bent Nielsen\footnote{Nuffield College \& Department of Economics, University of Oxford, United Kingdom. E-mail: \url{bent.nielsen@nuffield.ox.ac.uk}}
}
\date{February 13, 2025}

\maketitle

\begin{abstract}
We consider robust location-scale estimators under contamination. We show that commonly used robust estimators such as the median and the Huber estimator are inconsistent under asymmetric contamination, while the Tukey estimator is consistent. In order to make nuisance parameter free inference based on the Tukey estimator a consistent scale estimator is required. However,
standard robust scale estimators such as the interquartile range and the median absolute deviation are
inconsistent under contamination.
\end{abstract}


\section{Introduction}
In the early 1960s, there was an increasing awareness that standard estimators for normal models may not fare well under deviations from normality. 
\cite{huber1964} proposed  maximum likelihood-type ({M}) estimators for location and scale and found that they are more robust to such deviations than traditional estimators. Since then, robustness means that an estimate is only distorted in a bounded way when adding arbitrary observations to the sample \citep{hampel1971general}. While statistical inference theory was developed for data with infinitesimal contamination \citep{HeritierRonchetti1994}, little theory is available for cases with more contamination. 

We analyze popular M-estimators for location and scale when a fixed proportion of observations is contaminated. We find that the Huber estimator for location is not consistent, whereas the Tukey estimator for location has better consistency properties. To conduct inference, the uncontaminated scale has to be estimated.  Unfortunately, robust scale estimators such as the interquartile range (IQR) and the median absolute deviation (MAD) are inconsistent under contamination. Thus, valid statistical inference based on {M}-estimation seems to require detailed modelling of the contamination, which is typically not done in practice. 

Huber's article initiated the development of a large body of robust estimators, starting with the location-scale problem and later expanding to regression and many other statistical problems. An attractive feature of Huber's work was the idea of developing estimators that are not strongly based in statistical models. However, when thinking about inference, statistical models are useful. The lack of models is perhaps the reason that inference theory remains incomplete for the contaminated case. A major problem in the inference theory is that we need consistent estimators for both location and scale.  Location and scale are entwined. Without knowing one, it is hard to get to know the other. \cite{huber1964} was clearly aware of this.

Huber's $\epsilon$-contamination model mixes a normal distribution with an $\epsilon$ proportion of a contamination distribution. This model has become the standard contamination model. It is well suited for analyzing the consequences of infinitesimal contamination \citep{hampel1986robust, HeritierRonchetti1994, huber2009}; see \cite{avella2015robust, yu2017robust, loh2024theoretical} for recent reviews and \cite{HorowitzManski1995}
for a discussion of identification. But, when $\epsilon$ is not vanishing, standard scale estimators are inconsistent and asymptotic inference on the location depends on the contamination, which is usually unknown. Thus, inference for {M}-estimators is fraught under contamination.

Another type of robustness analysis uses the breakdown point \citep{hampel1971general}. The finite sample breakdown point asks how many arbitrary observations can be added to a sample without distorting the estimator unboundedly \citep{donohohuber1983,huber1984finite}. Location M-estimators typically have a high breakdown point. In practice, it is common to use {M}-estimators with a high breakdown point, which also maintain high efficiency under the normal model \citep{Yohai1987, coakley1993bounded, yu2017robust}.

Our approach is to consider a model where a proportion of the observations is contaminated. This matches the breakdown point idea. We consider two classes of {M}-estimators. The first class has an unbounded objective function. It includes the median and the Huber estimator. We find that these estimators are bounded in probability when more than half of the data is uncontaminated, which attests their robustness. Yet, these estimators are typically inconsistent under contamination. 

The second class has a bounded objective function and includes the Tukey estimator. We derive a lower bound on the proportion of contamination ensuring that these estimators are bounded in probability and consistent under contamination. 
However, nuisance parameter free inference requires a consistent scale estimator. We show that robust scale estimators such as the IQR and MAD are inconsistent under contamination.


These results for {M}-estimators are  disappointing.  In the discussion of the results, we 
point to the least trimmed squares (LTS) estimator \citep{rousseeuw1984least} which fares better.

We first revisit {M}-estimation and classical approaches to the robustness problem. Next, we introduce a new contamination model and analyse boundedness, consistency and asymptotic normality properties. This leads to new insights into the behaviour of {M}-estimators under contamination. The results are illustrated in a simulation study.

\section{Background}
Suppose we have scalar observations
$y_i$ for $i = 1, \dots, n$ and we want to estimate
a location-scale specification $y_i = \mu + \sigma \varepsilon_i$ where $\mu\in \mathbb{R}$ and $\sigma > 0$. A model for the errors $\varepsilon_i$ will be specified below.

\subsection{{M}-estimators for location with known scale}
The case of a known scale, $\sigma=1$ say, is the starting point for the {M}-estimation theory of \cite{huber1964}. The {M}-estimator $\hat\mu$ for location $\mu$
minimizes the objective function  
$
    \sum_{i=1}^n \rho
    \big( y_i - \mu 
    \big)
    .
$
We focus on $\rho$ functions that have a constant hence bounded derivative in the tails, as described in the following assumption. Note, this excludes the sample average with $\rho(x) = x^2$, which is not robust.

\begin{assumption}\label{assumption-rho}
    The function $\rho$ is defined on $\mathbb{R}$, positive apart from $\rho(0) = 0$, continuous, symmetric, and non-de\-crea\-sing in $x > 0$, and there are $x_*, \rho_* > 0$ and $\psi_* \geq 0$ such that $\rho(x) = \rho_* + \psi_*(|x| - x_*)$ for $|x| \geq x_*$.
\end{assumption}

We distinguish between two types of estimators. Non-redescending estimators have $\psi_* > 0$ and include
the median with $\rho(x) = |x|$ and the Huber estimator with $\rho(x) = (x^2/2) \mathbf{1}_{(|x| \leq c)} + c(|x| - c/2)\mathbf{1}_{(|x| > c)}$. Redescending estimators have $\psi_* = 0$ as for 
the Tukey estimator with 
$\rho(x) = (c^2/6)\{1 - (1 - x^2/c^2)^3\mathbf{1}_{(|x| \leq c)}\}$. 
The parameter $c > 0$ is a tuning parameter that the researcher needs to set in advance. In Section \ref{sec:tuning}, we discuss standard choices that are used in practice.


\subsection{M-estimators with unknown scale}\label{sec:unknown-scale}
In practice the scale $\sigma$ will be unknown. Suppose a preliminary estimator $\hat\sigma$ is available. Inserting this estimator in the above objective function for the location case
gives a new objective function
\begin{align}\label{for:loss-function-sigma-hat}
    R_n(\mu) = \sum_{i=1}^n \rho\left(\frac{y_i - \mu}{\hat\sigma}\right).
\end{align}
Note, the median computed from 
$\rho(x)=|x|$ is scale equivariant and is computable without knowing the scale.

It is important that the scale estimator is also robust, because otherwise the location estimator can still easily be affected by `outliers'. Standard robust scale estimators include the IQR and the MAD \citep[Chapter 5]{huber2009}. 
Letting $\hat{q}_p$ denote the $p$ quantile of the full sample, we write the estimators as
\begin{align}
     \hat\sigma_{\textsc{IQR}} = k  (\hat{q}_{3/4} - \hat{q}_{1/4}), \,
     \hat\sigma_{\textsc{MAD}} = k  \mathrm{med}\,\left|y_i - \hat{q}_{1/2} \right|.
     \label{scale_estimators}
\end{align}
The tuning parameters, $k$, must be set in advance by the researcher. 
Some standard choices are discussed in Section \ref{sec:tuning}.
We note that the combination of the median with the MAD is considered a simultaneous {M}-estimate of location and scale \citep[page 135]{huber2009}.

\subsection{The Huber contamination model}
\label{sec:huber}
The Huber $\epsilon$-contamination model \citep{huber1964} has become the standard way of thinking about and modelling contamination. In the $\epsilon$-contaminated normal model, the errors $\varepsilon_i$ are drawn from a mixture 
$F = (1-\epsilon)\Phi + \epsilon H$ with $\epsilon \in [0,1)$, $\Phi$ denotes the standard normal distribution and $H$ an unknown  contamination distribution. Theory often requires symmetric contamination. 

In this model, the asymptotic distribution of {M}-estima\-tors depends crucially on two unknown nuisance parameters: $\epsilon$ and  $H$. To get valid inference, the user will have to model the contamination. However, in practice, it is common to use the asymptotic theory under the normal model even when there is contamination.

\subsection{The infinitesimal approach}\label{sec:tuning}
In the infinitesimal approach the contamination proportion vanishes. This leads to the influence function theory, which describes the effect of an infinitesimal contamination of a point mass distribution on the estimator \citep{hampel1971general}. A condition for (local) robustness is that the influence function of an estimator is bounded.

In Huber's $\epsilon$-contamination model with 
vanishing $\epsilon$,
it is common to conduct inference as if the distribution is normal without any contamination. This eliminates the nuisance parameter problem mentioned in Section \ref{sec:huber}. Then, the nuisance parameters $c$ and $k$ are chosen from a perspective of efficiency in the uncontaminated normal model. Usual choices are $c = 1.345$ for the
Huber estimator and $c=4.685$ for the Tukey estimator, which give $95\%$ efficiency relative to the average \citep[pages 27, 30]{MaronnaMartinYohai2006}. 
The motivation is a trade-off between efficiency and robustness. The sample average is fully efficient but
can be arbitrarily distorted by one `outlier'. 
Thus,
a less efficient, but more robust estimator
is preferred. 
Tuning parameters that ensure consistency under normality 
are $k = 1/(q^\Phi_{3/4}-q^\Phi_{1/4})$ for IQR and $k = 1/q^\Phi_{1/4}$ for MAD, where $q^\Phi_p$ is the standard normal $p$-quantile \cite[page 36]{MaronnaMartinYohai2006}.

\subsection{The breakdown point}
\label{ss:breakdown_point}
The infinitesimal approach works well when the contamination proportion $\epsilon$ is very small. In empirical practice, this local property is not always a realistic assumption. Therefore, researchers are also interested in global robustness properties such as the breakdown point.

The finite sample breakdown point with $\epsilon$-contamina\-tion is defined as follows \citep{donohohuber1983}. Let $X$ be a sample of size $h$ of `good' observations. Adjoin $n-h$ arbitrary values $Y$ to get a sample $X\cup Y$ of size $n$. The sample $X\cup Y$ contains a fraction $\epsilon=(n-h)/n$ of arbitrary values. The breakdown point for an estimator $T$ is then $\epsilon^*(X,T)=\inf\{\epsilon:\sup_{Y}|T(X\cup Y)-T(X)|=\infty\}$. As $X$ is finite, so is $T(X)$ and breakdown happens if $Y$ diverges.

\cite{huber1984finite} analyzed the finite sample breakdown of {M}-estimators. For non-redescending estimators satisfying Assumption \ref{assumption-rho}, the breakdown point is $1/2$, while for redescending estimators, breakdown depends on the setup and the breakdown point is in general less than $1/2$.

\subsection{Some drawbacks of standard practice}\label{sec:standard}
The infinitesimal approach and the breakdown point approach are not compatible. The contamination proportion is vanishing for the former and non-vanishing for the latter. If one applies the breakdown point approach, then it is not a good idea to set the tuning parameters as if there is no contamination. For instance, it is not guaranteed that the M-estimator is consistent under contamination. Likewise, inference and efficiency considerations based on the normal model are not necessarily relevant.

Although the two approaches are not compatible, it has become standard practice to combine them. There is a focus on the use and development of robust estimators with high efficiency under the normal model and a high breakdown point \citep{Yohai1987, coakley1993bounded, GerviniYohai2002}. Moreover, it is typically advised not to use estimators that have a high breakdown point, but low asymptotic efficiency under the normal model \citep{yu2017robust}. 

Robustness and efficiency are often seen as unrelated problems. This may not be optimal. For example, even when an estimator has high efficiency under the normal model and a high breakdown point, then it is still not clear how that estimator behaves under contamination. Two estimators that satisfy both conditions, such as the Huber and Tukey estimators, might behave differently under contamination, and it unclear for a practitioner which estimator to choose. Yet, a high breakdown point estimator with a low asymptotic efficiency under the normal model may be preferable under contamination, see Section \ref{ss:LTS}.

\subsection{A different approach: modelling contamination}
The drawbacks mentioned in Section \ref{sec:standard} stem from the fact that there are no clear asymptotic results for M-estimators under contamination. Therefore, we analyse asymptotic properties such as boundedness, consistency and asymptotic normality of M-estimators under contamination. To this end, we introduce a simple contamination model inspired by the breakdown point theory in Section \ref{ss:breakdown_point}. We assume that a sample of observations $y_i$, with $i = 1, \dots, n$, has $h \leq n$ `good' observations and $n-h$ `outlier' observations. Let $\zeta_n$ be the set of $h$ elements from $(1, \dots, n)$ with indices of the `good' observations, while $\zeta_n^c$ denotes its complement, consisting of $n-h$ indices of `outliers'. Throughout Section \ref{sec:asymp-results}, we impose different assumptions on `good' and `outlying' observations.

\section{Asymptotic results}\label{sec:asymp-results}
\subsection{Asymptotic sequence of models}\label{sec:asymptotic}
We assume that a sample of observations $y_i$, with $i = 1, \dots, n$, has $h \leq n$ `good' observations and $n-h$ `outliers'.
The models are indexed by $n$, so that 
$h \rightarrow \infty$ as $n \rightarrow \infty$ 
and the proportion of `good' observations satisfies
$h/n \rightarrow \lambda$ with $1/2 < \lambda \leq 1$. 
Let $\zeta_n$ be a deterministic sequence of $h$-sets from $(1, \dots, n)$ 
of indices of `good' observations. The observations satisfy $y_i = \mu_\circ + \sigma_\circ \varepsilon_i$ 
for fixed $\mu_\circ, \sigma_\circ$.
Throughout, we specify sufficient conditions for errors $\varepsilon_i$ for `good' $i\in\zeta_n$ and `outliers' $j\in \zeta_n^c$.


\subsection{Boundedness}\label{sec:boundedness}
We show that {M}-estimators are bounded in probability in the asymptotic setup of Section \ref{sec:asymptotic}. The result requires a convergent scale estimator $\hat\sigma$. Ideally, the scale estimator should be consistent, but Assumption \ref{assumption-boundedness}$(i)$ allows general limits expressed in terms of a consistency factor $\varsigma$. Indeed, in Section \ref{sec:scale}, we show that standard robust scale estimators are typically inconsistent under contamination.

\begin{assumption}\label{assumption-boundedness}
    \hphantom{x}\\
    $(i)$: $\hat\sigma \overset{p}{\to} \varsigma\sigma_\circ$ for some $\varsigma>0$.\\
    $(ii)$: $h^{-1} \sum_{i \in \zeta_n} \varepsilon_i^2$ is bounded in probability. \\
    $(iii)$: $h^{-1} \sum_{i \in \zeta_n} \rho(\varepsilon_i/\alpha) \overset{p}{\to} \tilde{\rho}_\alpha$ uniformly over $\alpha$ near $\varsigma$, and where $\tilde{\rho}_\alpha$ is continuous in $\alpha$.
\end{assumption}

Assumption \ref{assumption-boundedness}$(iii)$ is a uniform law of large numbers. When the scale is known, then a standard law of large numbers suffices. Note, we do not need to make any assumptions on the `outliers'.

\begin{theorem}[Boundedness]\label{theorem-bounded}
    \hphantom{x}\\
    Suppose Assumptions \ref{assumption-rho}, \ref{assumption-boundedness} and that\\
    $(i)$: $\psi_* > 0$ and $\lambda > 1/2$; or \\
    $(ii)$: $\psi_* = 0$ and $\lambda > 1/(2 - \tilde{\rho}_{\varsigma}/\rho_*)$.\\
    Then, for large $n$, sets with large probability exist on which all minimizers $\hat\mu$ of $R_n(\mu)$ are uniformly bounded and at least one is measurable.
\end{theorem}


In the redescending case, Theorem \ref{theorem-bounded}$(ii)$ gives a lower bound to the proportion of `outliers' that is greater than a half and which depends on the $\rho$ function and the distribution of the `good' observations. This exactly matches the finding in \cite{donohohuber1983,huber1984finite} regarding the finite sample breakdown points of these estimators. We explore this connection.

The finite sample addition breakdown point for redescending {M}-estimators is $\epsilon^* = \lceil A \rceil/(h +\lceil A \rceil)$, where $A=h - \sum_{i\in\zeta_n} \rho\{(y_i - \hat\mu)/(\varsigma\sigma_\circ)\}/\rho_*$ \citep{huber1984finite}. If $\hat\mu \overset{p}{\to} \mu_\circ$, then $A/h$ and $1 - h^{-1} \sum_{i\in\zeta_n} \rho(\varepsilon_i/\varsigma)/\rho_*$ have the same limit $1 -\tilde{\rho}_{\varsigma}/\rho_*$, due to Assumption \ref{assumption-boundedness}$(iii)$. As a result, $1 -\epsilon^*\to 1/(2-\tilde{\rho}_{\varsigma}/\rho_*)$, as derived in Theorem \ref{theorem-bounded}$(ii)$.

Redescending M-estimators can tolerate more `outliers' when the scale is over-estimated as the lower bound to $\lambda$ in Theorem \ref{theorem-bounded}$(ii)$ is decreasing in $\varsigma$. However, when the scale is underestimated, then the redescending {M}-estimator can become unbounded. Indeed, as $\varsigma \rightarrow 0$, then $\tilde{\rho}_{\varsigma}/\rho_* \rightarrow 1$ so that $1/(2 - \tilde{\rho}_{\varsigma}/\rho_*) \rightarrow 1$. In Section \ref{sec:scale}, we show that the IQR and MAD typically overestimate in the presence of contamination. \cite{huber1984finite} also noted this in his finite sample breakdown point analysis. In Section \ref{sec:sim:boundedness}, we explore this in simulations.


\subsection{Asymptotics for redescending {M}-estimators}\label{sec:redescending}
We derive an oracle property for redescending {M}-estimators, such as the Tukey estimator. That is, the M-estimators for the contamination model are close to the minimizers of $\sum_{i\in\zeta_n} \rho \{(y_i-\mu)/(\sigma_\circ\varsigma)\}$ in which the set of `good' observations is known. This is helpful, as the latter problem has been studied in the literature.
\begin{assumption}\label{assumption-consistency}
    \hphantom{x}\\
    $(i)$: 'Outlier' errors:
    $1/(\min_{j \in \zeta_n^c} \varepsilon_j^2) \overset{p}{\to} 0$.\\
    $(ii)$: The function $\rho$ is Lipschitz: $\exists C>0$: $\forall x_1,x_2\in\mathbb{R}$: $|\rho(x_1)-\rho(x_2)|\le C |x_1-x_2|$.\\
    $(iii)$: 
    The minimizers of $\sum_{i\in\zeta_n}\rho\{(y_i-\mu)/(\sigma_\circ\varsigma)\}$ belong to a small neighbourhood of $\mu_\circ$ with large probability.
\end{assumption}

\noindent In Assumption \ref{assumption-consistency}, part $(i)$ requires that the smallest `outlier' errors diverges. This holds, for example, if the `good' errors are $\mathsf{N}(0,1)$, while  `outliers' are outside the range of the `good' errors. Part $(ii)$ requires that $\rho$ is Lipschitz, which holds for the Tukey estimator. Part $(iii)$ requires that the infeasible {M}-estimator is close to $\mu_\circ$. 

\begin{theorem}[Closeness]\label{theorem-consistency}
    \hphantom{x} \\
    Suppose Assumptions \ref{assumption-rho}, \ref{assumption-boundedness}, \ref{assumption-consistency}, while $\psi_* = 0$ and $\lambda > 1/(2 - \tilde{\rho}_\varsigma/\rho_*)$. Then, for large $n$, large probability sets exist on which all minimizers of $R_n(\mu)$ are near $\mu_\circ$, and at least one is measurable and therefore consistent.
\end{theorem}


We can get an asymptotic distribution for the $M$-estimator under some additional assumptions allowing asymmetric contamination.

\begin{assumption}\label{assumption-distribution}
    \hphantom{x}
    \\ $(i)$ $n^{1/2}(\hat\varsigma-\varsigma)$ is bounded in probability.
    \\ $(ii)$ $\varepsilon_i$, $i\in\zeta_n$, are i.i.d.; $\mathsf{E}\ddot\rho(\varepsilon_i/\varsigma)>0$, $\mathsf{E}\varepsilon_i\ddot\rho(\varepsilon_i/\varsigma)=0$ and $\mathsf{Var}\{\varepsilon_i\ddot\rho(\varepsilon_i/\varsigma)\}<\infty$.
    \\ $(iii)$ $\rho$ is has a Lipschitz continuous second derivative.    
    \\ $(iv)$ $\sum_{i\in\zeta_n}\rho\{(y_i-\mu)/(\sigma_\circ\varsigma)\}$ has a measurable mini\-mi\-zer $\hat\mu_{\zeta_n}$ so that $h^{1/2}(\hat\mu_{\zeta_n} - \mu_\circ)\overset{d}{\to}\mathsf{N}(0, \sigma_\circ^2 V_{\varsigma})$, for a $V_{\varsigma} > 0$. Any other minimizer  
    belongs to a neighbourhood of $\hat\mu_{\zeta_n}$ that shrinks faster than $n^{-1/2}$ with large probability.
\end{assumption}

\begin{theorem}[Asymptotic distribution]\label{theorem-distribution}
    \hphantom{x} \\
    Suppose Assumptions \ref{assumption-rho}, \ref{assumption-boundedness}, \ref{assumption-consistency}, \ref{assumption-distribution}, while $\psi_* = 0$ and $\lambda > 1/(2 - \tilde{\rho}_\varsigma/\rho_*)$. Then, for sufficiently large $n$, there exist high-probability sets on which all minimizers lie in a neighbourhood of $\hat\mu_{\zeta_n}$ shrinking faster than $n^{-1/2}$. 
    \\ A measurable minimizer $\hat\mu$ exists so that $h^{1/2}(\hat\mu-\mu_\circ)$ is asymptotically $\mathsf{N}(0, \sigma_\circ^2 V_{\varsigma})$. 
\end{theorem}

\noindent Sufficient conditions for Assumption \ref{assumption-distribution} are given in \cite{jureckova1996robust}, Chapter 5.3. In particular, if the `good' errors are independent standard normal and $\rho$ is twice differentiable, then $V_{1} = \int_{\mathbb{R}} \{\dot\rho (x)\}^2 d\Phi(x)/\{\int_{\mathbb{R}} \ddot\rho(x) d\Phi(x)\}^2$. %

Theorem \ref{theorem-distribution} gives a basis for discussing asymptotic efficiency within the class of robust estimators with the oracle property that $\hat\mu-\hat\mu_{\zeta_n}$ is small. For example, if the `good' observations are all $\mathsf{N}(0,1)$, then the benchmark is the fully efficient average of the $h$ `good' observations. Redescending {M}-estimators have efficiency $1/V_{\varsigma}$. In particular, the Tukey estimator with known scale and $c=4.685$ achieves 95\% efficiency. Thus, the Tukey estimator achieves high efficiency in both the normal model without contamination \cite[page 30]{MaronnaMartinYohai2006} 
and in the normal model with contamination. This does, however, require a consistent estimator of the scale $\sigma_{\circ}$.


\subsection{Inconsistency of non-redescending {M}-estimators}
Non-redescending {M}-estimators, such as the median and the Huber estimator, are inconsistent in the presence of asymmetric contamination. We illustrate this using a class of data generating processes matching the setup for finite sample breakdown point considerations.



\begin{assumption}\label{assumption-non-consistency}
    $(i)$: $\hat\sigma=\sigma_\circ$, \\
    $(ii)$: `Good' errors are independent standard normal. \\
    $(iii)$: $h^{-1}\sum_{i\in\zeta_n}\rho(\varepsilon_i-\mu)\to \mathsf{E}\rho(\varepsilon_1-\mu)$ in probability, uniformly over a $\mu$-neighbourhood of zero. \\
    $(iv)$: $\mathsf{E}\rho(\varepsilon_1 - \mu)$ has $\mu$-derivative of $0$ at $\mu = 0$ for $i \in \zeta_n$. \\
    $(v)$: `Outlier' errors: $\varepsilon_j = \max_{i \in \zeta_n} \varepsilon_i + \xi$ for some $\xi>0$. 
\end{assumption}
\noindent
Note, Assumption \ref{assumption-non-consistency}$(i)$-$(iii)$ implies Assumption \ref{assumption-boundedness}, so that the boundedness Theorem \ref{theorem-bounded} applies. Part $(v)$ provides a particular asymmetric contamination.

\begin{theorem}\label{theorem-non-consistency}
    Suppose Assumptions \ref{assumption-rho}, \ref{assumption-non-consistency} and $\psi_* > 0$, $1/2 < \lambda < 1$. Then $\hat\mu$ is inconsistent for $\mu_\circ$.
\end{theorem}


\subsection{Scale estimation}\label{sec:scale}

Theorems \ref{theorem-bounded}$(ii)$ and \ref{theorem-consistency} depend on the scale estimator $\hat\sigma$ via the consistency factor $\varsigma$. Thus, it is important to find robust scale estimators that are bounded and consistent under contamination. 
We first show that 
the IQR and 
MAD defined in (\ref{scale_estimators}) are bounded in probability.

\begin{assumption}    
    \label{assumption-iqr-bounded} 
    The $p$ quantiles of $\varepsilon_i$ for $i\in\zeta_n$ are bounded in probability for $0<p<1$.
\end{assumption}
\begin{theorem}
    \label{theorem-iqr} Suppose Assumption \ref{assumption-iqr-bounded}.\\
    $(a)$: Let $\lambda > 3/4$. Then, $ \hat\sigma_{\textsc{IQR}}$ is bounded in probability.\\
    $(b)$: Let $\lambda > 1/2$. Then $\hat\sigma_{\textsc{MAD}}$ is bounded in probability.
\end{theorem}
Theorem \ref{theorem-iqr} shows that the IQR is bounded when the asymptotic proportion of `good' observations $\lambda > 3/4$, while the MAD requires $\lambda > 1/2$. This aligns with their finite sample breakdown points of $1/4$ for the IQR and $1/2$ for the MAD \citep[page 106]{huber2009}.

Next, we show that the two scale estimators are, in general, inconsistent under contamination. We normalize the estimators to be consistent for normal data, see Sections \ref{sec:unknown-scale}, \ref{sec:tuning}, but consider `good' errors with a general distribution and allow contamination.
\begin{assumption}
    \label{assumption-igr-inconsistency} $(i)$: `Good' errors $\varepsilon_i$, $i\in\zeta_n$ are independent with a continuous distribution $\mathsf{F}$. \\
    $(ii)$: `Outliers' are extreme: $\varepsilon_j > \max_{i\in\zeta_n} |\varepsilon_i| $, $j\in\zeta_n^c$. \\
    $(iii)$: Proportion of left `outliers': $n^{-1}\sum_{j\in\zeta_n^c} 1_{(\varepsilon_j<0)} \overset{p}{\to}\varrho$, where $0\le\varrho\le 1-\lambda$.
\end{assumption}

We define the consistency factors 
\begin{align}
    \varsigma_{\textsc{IQR}} = \frac{q^{\mathsf{F}}_{(3/4 - \varrho)/\lambda} - q^{\mathsf{F}}_{(1/4 - \varrho)/\lambda}}{q^\Phi_{3/4} - q^\Phi_{1/4}}, \quad \varsigma_{\textsc{MAD}} = \frac{d}{q^\Phi_{3/4}},
    \label{scale-consistency-factors}
\end{align}
where $d$ solves $\mathsf{F}(c + d) - \mathsf{F}(c - d) = 1/(2\lambda)$ for a $c$ solving $\mathsf{F}(c) = (1/2 - \varrho)/\lambda$, where $q^{\mathsf{F}}_p$ and $q^\Phi_p$ are the $\mathsf{F}$ and normal $p$ quantiles. When the `good' errors are normal and $\lambda=1$, so that $\varrho = 0$, then $\varsigma_{\textsc{IQR}}=\varsigma_{\textsc{MAD}}=1$. In the Appendix, we show that when Assumption \ref{assumption-igr-inconsistency} holds with `good' normal errors and $\lambda<1$, then both consistency factors are greater than unity, see Remarks \ref{rem:consistency_IQR}, \ref{rem:consistency_MAD}.

\begin{theorem}\label{prop-inconsistent-scale}
    Suppose Assumption \ref{assumption-igr-inconsistency}. Then, \\
    $(a)$: If $3/4 < \lambda \le 1$, then $\hat\sigma_{\textsc{IQR}} \overset{p}{\rightarrow}\varsigma_{\textsc{IQR}} \sigma_\circ$.\\
    $(b)$: If $1/2 < \lambda \le 1$, then $\hat\sigma_{\textsc{MAD}} \overset{p}{\rightarrow}\varsigma_{\textsc{MAD}} \sigma_\circ$.
\end{theorem}

The scale estimators $\hat\sigma_{\textsc{IQR}}$ and $\hat\sigma_{\textsc{MAD}}$ are, in general, inconsistent under contamination. For the non-redescen\-ding median and Huber estimator, boundedness is not affected by the inconsistency of the scale (Theorem \ref{theorem-bounded}). However, these estimators are inconsistent under asymmetric contamination (Theorem \ref{theorem-non-consistency}). For redescending estimators, such as the Tukey estimator, the situation is more subtle. 
The boundedness result can tolerate more contamination when the scale is overestimated (Theorem \ref{theorem-bounded}), which it tends to be (Remarks \ref{rem:consistency_IQR}, \ref{rem:consistency_MAD}). This in line with the breakdown point 
analysis in \cite{huber1984finite}. Further, inconsistency of the scale estimator does not affect the consistency of the location estimator. However, it does result in nuisance parameters when conducting inference (Theorem \ref{theorem-consistency}). In summary, conducting valid inference with {M}-estimators seems to require detailed modelling of the contamination.

\section{Simulation study}\label{sec:simulation}
We study finite sample properties of four {M}-estimators of location (average, median, Huber, Tukey). We consider bias, the effect of scale estimation and boundedness.

\subsection{Benchmark: Least Trimmed Squares}\label{ss:LTS}
In Section \ref{sec:redescending}, we used the average, applied infeasibly to the set $\zeta_n$ of $h$ `good' observations, as a benchmark for efficiency. In practice, the set $\zeta_n$ is unknown and must be estimated. The Least Trimmed Squares (LTS) estimator \citep{rousseeuw1984least} does this and it is asymptotically efficient.
 
The LTS estimator requires that the user specifies there are $h$ `good' observations and $n-h$ `outliers'. The LTS estimator minimizes the residual sum of squares over all $h$-subsets of the data. In the location-scale case, we can proceed as follows. Let $\zeta$ be a $h$-subset of $(1,\dots,n)$ with indices of the `good' observations. This is estimated by
$$
    \hat\zeta = \mathrm{argmin}\,_{\zeta} \sum_{i \in \zeta} (y_i - \bar{y}_{\zeta})^2
    ,
$$
while the estimators of location and scale are
\begin{align*}
    \hat\mu_{\textsc{LTS}} = \bar{y}_{\hat\zeta}, \qquad \hat\sigma_{\textsc{LTS}}^2 &= \frac1h\sum_{i \in \hat\zeta}(y_i - \hat\mu_{\textsc{LTS}})^2.
\end{align*}

The LTS estimator is maximum likelihood in a contamination model satisfying Assumption \ref{assumption-boundedness}, where `good' observations are i.i.d.\ normal and `outlier' errors are more extreme than the `good' errors as in Assumption \ref{assumption-consistency}$(i)$, but otherwise unspecified \citep{berenguer2023model}. This LTS model matches the setup of the finite sample breakdown point, see Section \ref{ss:breakdown_point}. Indeed, the LTS estimator has a breakdown point of 50\% and this result extends to a regression version of the estimator \citep{RousseeuwLeroy1987}. When applied to a clean normal sample, then the LTS estimator applied with $h\approx n/2$ has an efficiency of about 7\% relative to the full sample average (\citealt{Butler1982}; \citealt[page 180]{RousseeuwLeroy1987}). 
However, under the LTS model, the standardized estimator $h^{1/2}(\hat\mu_{LTS}-\mu_\circ)/\hat\sigma_{LTS}$ has the oracle property that it equals the standardized average with large probability, so that it is asymptotically standard normal under the LTS model \citep{berenguer2023model}. This result generalizes to a regression model that allows leverage effects \citep{berenguer2023model,berenguer2023leverage}.  

In practice, the user will have to estimate the proportion $\lambda$ of `good' observations by estimating $h/n$. A standard method to estimate $\lambda$ is the index plot method \citep[page 55]{RousseeuwLeroy1987}. 
A detailed analysis is likely to show that this method is not consistent and therefore will not result in a fully efficient LTS estimator. 
Instead, we estimate $\lambda$ using the cumulant based method of \cite{berenguer2023model}. This estimator is consistent, but it remains work in progress to show that the resulting LTS estimator has the oracle property.

\subsection{Simulation design}
The location estimators are as follows. At first, we consider the scale to be known. The sample average is used as a non-robust benchmark. Neither the average, the median nor the LTS estimator require initial scale estimation. The Huber and Tukey estimators require initial scale estimation and are used with tuning parameters $c = 1.345$ and $c = 4.685$, respectively, to achieve $95\%$ efficiency. The LTS estimator requires a user choice of $h$ to set trimming.

At first, we let the scale $\sigma_{\circ}$ be known and choose $h$ as 80\% of $n$ in the LTS estimation. Subsequently, we estimate $\sigma_{\circ}$ with the MAD and we estimate $h$ using a cumulant based normality test \citep{berenguer2023model}.

We consider six data generating processes (DGPs). For all DGPs we use $\mu_{\circ} = 0$ and $\sigma_{\circ} = 1$, consider sample sizes $n = 25, 100, 400$ and use $10^5$ repetitions. 

DGP1--2 have no contamination. DGP1 has independent $\mathsf{N}(0,1)$ errors and DGP2 has independent $t(3)$ errors.

DGP3--6 have contamination: DGP3--5 have $h/n = \lambda = 0.8$; DGP6 has $h=n-1$. 
The `good' errors are independent $\mathsf{N}(0,1)$ for DGP3,4,6 and independent $t(3)$ for DGP5.
The `outlier' errors satisfy $\varepsilon_j = \max_{i \in \zeta_n} \varepsilon_i + \xi$, $j \in \zeta_n^c$. We have $\xi = 1$ for DGP3 and $\xi = 3$ for DGP4--6. 
DGP6 mimics infinitesimal $\epsilon$-contamination.

\subsection{Bias}

Tables \ref{table-sigma-known}--\ref{table-sigma-unknown} report bias for different estimators when scale and trimming are, respectively, known and estimated. This only affects the Huber, Tukey and LTS estimators. As $10^5$ repetitions are used, the maximal Monte Carlo standard error was below $0.0005$.

\begin{table}[tb]
    \caption{Bias for known scale and trimming.}
            \centering
        \begin{tabular}{crrrcrrr}
           \hline
                               & DGP1    & DGP2      & DGP3    & DGP4     & DGP5      & DGP6      \\
            \hline
            $n = 25$    &&&&&&\\
            Mean   & $0.000$ & $0.002$   & $0.573$ & $0.973$  & $1.282$   & $0.198$   \\
                         Median & $0.001$ & $0.001$   & $0.315$ & $0.315$  & $0.353$   & $0.052$   \\
                         Huber  & $0.001$ & $0.001$   & $0.412$ & $0.412$  & $0.468$   & $0.067$   \\
                         Tukey  & $0.001$ & $0.000$   & $0.394$ & $0.011$  & $0.003$   & $0.001$   \\
                         LTS    & $0.000$ & $0.002$   & $0.061$ & $0.000$  & $0.079$   & $0.000$   \\
            \\                                         
            $n = 100$   &&&&&&\\
            Mean   & $0.000$ & $0.001$   & $0.685$ & $1.085$  & $1.754$   & $0.055$   \\
                         Median & $0.000$ & $0.000$   & $0.317$ & $0.317$  & $0.350$   & $0.012$   \\
                         Huber  & $0.001$ & $0.000$   & $0.414$ & $0.414$  & $0.469$   & $0.016$   \\
                         Tukey  & $0.001$ & $0.000$   & $0.305$ & $0.001$  & $0.000$   & $0.001$   \\
                         LTS    & $0.000$ & $0.001$   & $0.004$ & $0.001$  & $0.043$   & $0.001$   \\
            \\                                         
            $n = 400$   &&&&&&\\
            Mean   & $0.000$ & $0.000$   & $0.780$ & $1.180$  & $2.485$   & $0.015$   \\
                        Median & $0.000$ & $0.000$   & $0.318$ & $0.318$  & $0.320$   & $0.003$   \\
                        Huber  & $0.000$ & $0.000$   & $0.415$ & $0.415$  & $0.469$   & $0.004$   \\
                        Tukey  & $0.000$ & $0.000$   & $0.176$ & $0.000$  & $0.000$   & $0.000$   \\
                        LTS    & $0.000$ & $0.000$   & $0.001$ & $0.000$  & $0.024$   & $0.000$   \\
            \hline            
        \end{tabular}
     \label{table-sigma-known}
\end{table}
\begin{table}[tb]
\caption{Bias for unknown scale and trimming.}
            \centering
        \begin{tabular}{crrrcrrr}
           \hline
                                & DGP1    & DGP2      & DGP3    & DGP4     & DGP5      & DGP6      \\ 
            \hline                    
            $n = 25$    &&&&&&\\
                        Huber  & $0.001$ & $0.001$   & $0.486$ & $0.493$  & $0.624$   & $0.068$   \\
                         Tukey  & $0.001$ & $0.000$   & $0.480$ & $0.306$  & $0.357$   & $0.010$   \\
                        LTS    & $0.002$ & $0.001$   & $0.018$ & $0.001$  & $0.036$   & $0.000$   \\
            \\                                         
            $n = 100$   &&&&&&\\
                        Huber  & $0.001$ & $0.000$   & $0.499$ & $0.499$  & $0.626$   & $0.016$   \\
                        Tukey  & $0.001$ & $0.000$   & $0.516$ & $0.196$  & $0.101$   & $0.000$   \\
                        LTS    & $0.001$ & $0.000$   & $0.002$ & $0.000$  & $0.000$   & $0.001$   \\
            \\                                         
            $n = 400$   &&&&&&\\
                        Huber  & $0.000$ & $0.000$   & $0.501$ & $0.501$  & $0.626$   & $0.004$   \\ 
                        Tukey  & $0.000$ & $0.000$   & $0.505$ & $0.087$  & $0.003$   & $0.000$   \\
                        LTS    & $0.000$ & $0.000$   & $0.002$ & $0.000$  & $0.000$   & $0.000$   \\ 
            \hline            
        \end{tabular}
     \label{table-sigma-unknown}
\end{table}

DGP1 has no contamination and normal errors so that all methods perform well. In this case, the MAD is consistent for $\sigma_{\circ}$. Thus, even with unknown scale, we see good performance of the Huber and Tukey estimators. \cite{berenguer2023model} give simulation evidence that the cumulant based estimator for $h$ is consistent when the `good' errors are normal, so that the LTS estimator also performs well when estimating $h$.

DGP2 has no contamination and $t(3)$ errors. For this reason, all estimators are consistent. When $\sigma_{\circ}$ is unknown, the MAD overestimates, which does not negatively affect the consistency properties of the Huber and Tukey estimator as there is no contamination. The LTS estimator performs very well, even 
for unknown $h$. 

DGP3--DGP4 have normal `good' errors and extreme `outliers'. The bias of the average grows as the sample size grows because the `outliers' are larger for larger sample sizes. Moreover, the bias is increasing in the `outlier' parameter $\xi$. The bias of the median and the Huber estimator remain constant when $n$ and $\xi$ grow, which shows their robustness. They are, however, inconsistent (Theorem \ref{theorem-non-consistency}). When $\sigma_{\circ}$ is unknown, the MAD overestimates (Theorem \ref{prop-inconsistent-scale} and Remark \ref{rem:consistency_MAD}). Overestimation of the scale results in less downweighting of `outliers', hence more bias. The Tukey estimator is consistent so that the bias decreases as the sample size grows (Theorem \ref{theorem-consistency}). When $\sigma_{\circ}$ is known, the Tukey estimator does not always penalize the `outliers' in DGP3, so the bias is large in small samples. When $\sigma_{\circ}$ is unknown, the performance of the Tukey estimator worsens and convergence is very slow for DGP3. Extra simulations showed biases of $0.458$ and $0.386$ for $n = 1600, 6400$, respectively. The LTS estimator performs very well, even when $h$ is unknown. 

DGP5 has $t(3)$ `good' errors and extreme 'outliers'. As in DGP3--4 the average, median and Huber estimators are biased. The Tukey estimator performs very well. This happens, because the `outliers' are generated by the maximal `good' errors and tend to be larger than the `outliers' in DGP4. These `outliers' are down-weighted, which results in a good performance of the Tukey estimator. With unknown scale, the MAD overestimates, which results in slower convergence than in the known scale scenario. The LTS estimator performs very well, and even better when $h$ is unknown. This happens because $h$ is estimated using a cumulant based normality test, while the `good' observations are $t(3)$, resulting in a beneficial underestimation.

DGP6 mimics the situation of infinitesimal contamination. We see that all estimators are consistent as the bias decreases to zero as sample size increases. This happens, because the influence of the `outlier' diminishes as the samples size increases. We do, however, see that the average is quite biased for smaller sample sizes. In particular, if $\xi \rightarrow \infty$, then the bias of the average would diverge, while the bias of the other estimators would be unaffected.

\begin{figure}[tb]
    \centering
    \includegraphics[scale = 0.75]{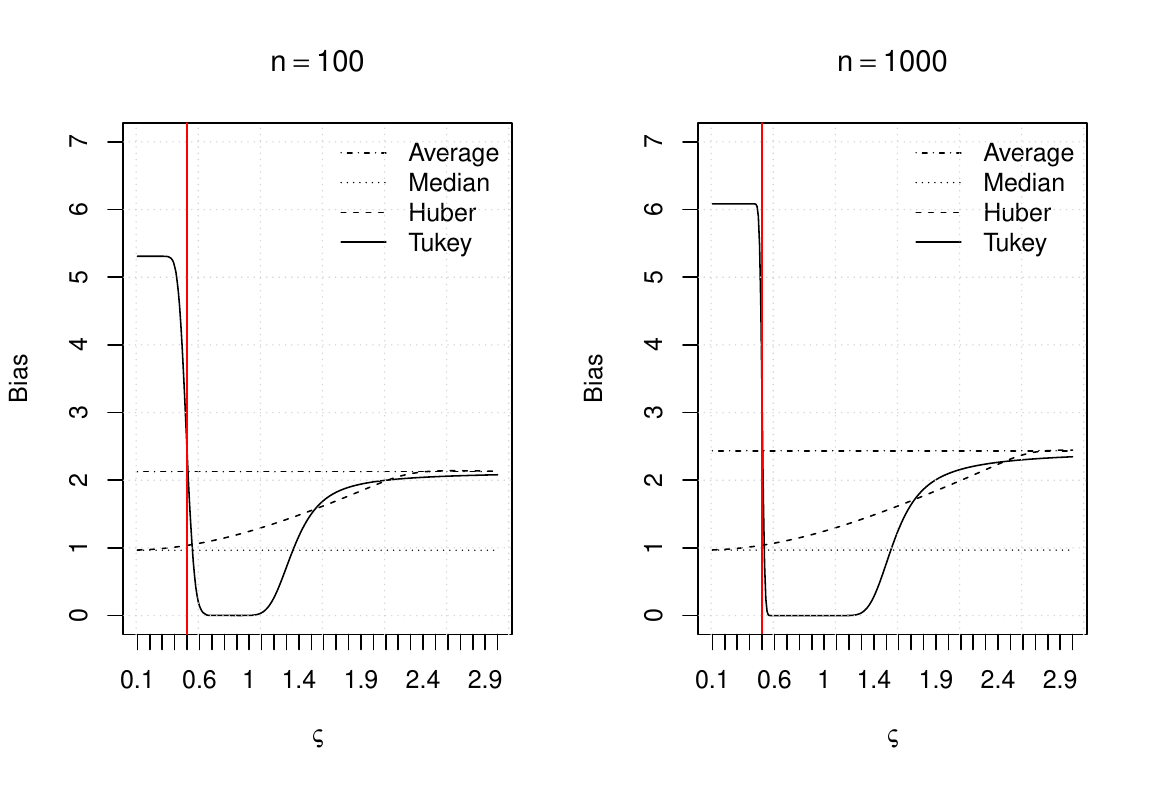}
    \caption{Bias of different {M}-estimators in relation to the consistency factor $\varsigma$. The red vertical line indicates the smallest 
    $\varsigma$ for which boundedness holds for Tukey's estimator according to 
    Theorem \ref{theorem-bounded}$(ii)$.}
    \label{fig:boundedness-n100}
\end{figure}

\subsection{Boundedness}
\label{sec:sim:boundedness}
We explore how the bias of 
four 
{M}-estimators depends on the scale estimation. 
We use $\hat\sigma=\varsigma\sigma_\circ$ and vary the consistency factor $\varsigma$. Figure \ref{fig:boundedness-n100} has two graphs showing the bias of the estimator for varying
$\varsigma$ and
$n=100$, $1000$.

We generate data from
DGP4, albeit with $h/n = 0.6$. The Huber and Tukey estimators 
have the same tuning parameters, tuned for $\sigma_{\circ} = 1$, as before.

The average and median are biased, but constant in $\varsigma$ due to their scale invariance. For large $n$, the median approaches the $5/6$ normal quantile, which is 0.97. The average diverges as $0.4 (\max_{i\in\zeta_n}\varepsilon_i+3)$. We have $\max_{i\in\zeta_n}\varepsilon_i/\{2\log(0.6n)\}^{1/2}\overset{p}{\to} 1$, see \cite{berenguer2023leverage}, Appendix B.

Consider the Tukey estimator. When $\varsigma = 1$, the estimator has zero bias. When $\varsigma > 1$, the bias is zero up until $\varsigma = 1.2\ (n = 100)$ or $\varsigma = 1.3\ (n = 1000)$ and converges towards the bias of the average. When $\varsigma$ is large, the `outliers' are not down-weighted so that the Tukey estimator behaves as the average. Nevertheless, it is consistent since the `outliers' diverge and therefore for any fixed $\varsigma > 1$ the `outliers' are, at some point, down-weighted. When $\varsigma < 1$, the bias is zero down to $\varsigma = 0.65\ (n=100)$ or $\varsigma = 0.55\ (n = 1000)$. Below those values, the Tukey estimator hones in on the outliers so that its bias diverges as $\max_{i\in\zeta_n}\varepsilon_i+3$.
This and the dependency on $\xi$ shows unboundedness of the estimator. The red vertical line in Figure \ref{fig:boundedness-n100} indicates the smallest value of $\varsigma$ for which boundedness holds according to Theorem \ref{theorem-bounded}$(ii)$. The simulations indicate that the first part of the bias curve converges to a function with a step at the red line as expected.

Consider the Huber estimator. When $\varsigma \rightarrow 0$, it behaves as the median as the width of the central quadratic part of the objective function shrinks, so that the linear parts dominate. When $\varsigma \rightarrow \infty$, it behaves as the average as the quadratic part of the objective function now dominates.

\section{Discussion}
Robustness is the property that an estimator changes in a bounded way when adding contaminated observations to a sample. This is traditionally measured in terms of the finite sample breakdown point. {M}-estimators are known to have a high finite sample breakdown point for location-scale models \citep{donohohuber1983,huber1984finite}. The  
results on asymptotic boundedness confirm this for location estimators (Theorem \ref{theorem-bounded}) and for scale estimators (Theorem \ref{theorem-iqr}). These results do not generalize to regression M-estimators \citep{he1990tail}.

Robust estimators that are consistent and have simple inferential properties under contamination are particularly desirable. Theorems \ref{theorem-consistency}, \ref{theorem-non-consistency} showed that redescending {M}-estimators are consistent under contamination, while non-redescending {M}-estimators are not. Even so, simple inference based on redescending M-estimators requires a 
scale estimator that is consistent under contamination. This fails for 
robust scale estimators such as the IQR and 
MAD
(Theorem \ref{prop-inconsistent-scale}). It remains an open problem to find robust 
{M}-estimators with simple inference when a proportion of the observations are contaminated.

By simulation, we found that the LTS estimator performs better than various {M}-estimators, making it an attractive alternative. LTS estimators have a high breakdown point in location-scale models and in regression \citep[pages 132 -- 135]{RousseeuwLeroy1987}. 
LTS estimators for location and scale, and more generally for regression coefficients, have the oracle property that inference is the same as for the infeasible least squares estimator on the `good' observations only \citep{berenguer2023model,berenguer2023leverage}. Inference is, therefore, nuisance parameter free and it achieves full efficiency in the contaminated model (Section \ref{sec:redescending}). These results do, however, require that the proportion of `good' observations is known. One could say that the problem of estimating the scale has been replaced by the problem of estimating the proportion of the `good' observations. This appears to be an easier problem to solve theoretically. Indeed, an estimator proposed in \cite{berenguer2023model} appeared successful in the simulations. 

\appendix
\section{Proofs}
We consider estimators for $\mu$ that are location-scale invariant and estimators for $\sigma$ that location invariant and scale equivariant. Thus, throughout we set $\mu_\circ = 0$ and $\sigma_\circ = 1$, so that $y_i = \varepsilon_i$ and $\hat\varsigma = \hat\sigma$.
\begin{proof}[Proof of Theorem \ref{theorem-bounded}]
    The proof is inspired by \cite{berenguer2023leverage,johansen2019boundedness}. 
    Let $e^\mu_i=(y_i-\mu)/\hat\sigma=(\varepsilon_i-\mu)/\hat\varsigma$ so that $R_n(\mu)=\sum_{i=1}^n\rho(e^\mu_i)$.
    We show that $\{R_n(\mu)-R_n(0)\}/h\ge L>0$  for large $\mu$ on a large probability set. Thus, no minimizer is large.

    Consider large $|\mu|>\hat\varsigma B_0$ for a $B_0>0$ to be chosen. Splitting the sum in $R_n(\mu)$ over $i$ in sums over $\zeta_n$ and $\zeta_n^c$ while truncating by an $A_0>1$ to be chosen gives  
    $$  R_n(\mu) \ge \underset{{j\in\zeta_n^c}}{\textstyle\sum} \rho(e^\mu_j) + \underset{{i\in\zeta_n}}{\textstyle\sum}\rho(e^\mu_i)1_{(|\varepsilon_i|\le  A_0)}.    $$ 
    By the reverse triangle inequality, $|e^\mu_i|\ge (|\mu|-|\varepsilon_i|)/\hat\varsigma$. We get  $|e^\mu_i|>B_0- A_0/\hat\varsigma$ when $|\varepsilon_i|\le A_0$ as $|\mu|>\hat\varsigma B_0$. Further, $|e^\mu_i|>x_*$ if $B_0 \ge x_*+ A_0 /\hat\varsigma$. By Assumption \ref{assumption-rho}, we get $\rho (e^\mu_i) = \rho_* + \psi_*(|e^\mu_i| - x_*)$. Thus, 
    $$  R_n(\mu) \ge \underset{{j\in\zeta_n^c}}{\textstyle\sum}\rho(e^\mu_j)
        + \underset{{i\in\zeta_n}}{\textstyle\sum}\{\rho_* + \psi_*(|e^\mu_i| - x_*)\}1_{(|\varepsilon_i|\le  A_0)}.    $$ 
    As $R_n(0)=\sum_{j\in\zeta_n^c}\rho(\varepsilon_j/\hat\varsigma)+\sum_{i\in\zeta_n}\rho(\varepsilon_i/\hat\varsigma)$, we get
    \begin{multline}
        R_n(\mu)-R_n(0)\ge \underset{{j\in\zeta_n^c}}{\textstyle\sum}\{\rho(e^\mu_j)-\rho(\varepsilon_j/\hat\varsigma)\}
    \\  + \underset{{i\in\zeta_n}}{\textstyle\sum}\big[\{\rho_* + \psi_*(|e^\mu_i| - x_*)\}1_{(|\varepsilon_i|\le  A_0)}-\rho(\varepsilon_i/\hat\varsigma)\big].\label{for-lemma-boundedness-sigmahat}
    \end{multline}    
   
    We bound the second term in (\ref{for-lemma-boundedness-sigmahat}). Argue as above to show $|e^\mu_i|\ge (|\mu|-A_0)/\hat\varsigma$. Further, we bound $1_{(|\varepsilon_i| \le A_0)}\ge 1 - \varepsilon_i^2/A_0^2$ such that $h^{-1}\sum_{i\in\zeta_n}1_{(|\varepsilon_i|\le A_0)}\ge 1- C_n/ A_0^2$, with $C_n = h^{-1}\sum_{i \in \zeta_n} \varepsilon_i^2$. Let $\hat{\rho}_{\hat\varsigma} = h^{-1} \sum_{i \in \zeta_n} \rho(\varepsilon_i/\hat\varsigma)$. The second term is then at least $h[\{\rho_* + \psi_*(|\mu|/\hat\varsigma - A_0/\hat\varsigma - x_*)\}(1-C_n/A_0^2) - \hat{\rho}_{\hat\varsigma}]$.
        
    We bound the summands $r^\mu_j =\rho ( e^\mu_j )  - \rho(\varepsilon_j/\hat\varsigma)$ in the first term of (\ref{for-lemma-boundedness-sigmahat}).  We show that $r^\mu_j > -\rho_* - \psi_* |\mu|/\hat\varsigma$. Since $\rho$ is symmetric, it suffices to consider $\mu > 0$ and $\varepsilon_j \in \mathbb{R}$. We check two main cases, each with three sub-cases.
        
    Case 1. Let $\mu > 0$ and $ \varepsilon_j \ge 0$. Then $ \varepsilon_j > \varepsilon_j - \mu$. 
    \\(a) If $ \varepsilon_j/\hat\varsigma >  (\varepsilon_j - \mu)/\hat\varsigma \ge x_*$, then $\rho$  is linear in both arguments, so that $r^\mu_j =\psi_*(\varepsilon_j - \mu)/\hat\varsigma - \psi_*\varepsilon_j/\hat\varsigma =-\psi_* \mu/\hat\varsigma$.
    \\(b) If $\varepsilon_j/\hat\varsigma \ge x_* > (\varepsilon_j - \mu)/\hat\varsigma $, then the first $\rho$ is at least 0 and the second is linear and equal to $\rho_* + \psi_*(\varepsilon_j/\hat\varsigma - x_*)$. Exploit that $x_* > (\varepsilon_j - \mu)/\hat\varsigma$ to get $r^\mu_j >  -\rho_* - \psi_*\mu/\hat\varsigma.$ 
    \\(c) If $x_* \ge \varepsilon_j/\hat\varsigma \ge 0$, then the first $\rho$ is at least 0 and the second is at most $\rho^*$, so that $r^\mu_j \ge 0 - \rho_* = -\rho_*$.

    Case 2. Let $\mu > 0$ and $\varepsilon_j < 0$. Then $0 > \varepsilon_j >  \varepsilon_j - \mu$. 
    \\(a) If $\varepsilon_j/\hat\varsigma >  (\varepsilon_j - \mu)/\hat\varsigma \ge -x_*$, then the first $\rho$ is at least 0 and the second is at most $\rho^*$, so that $r^\mu_j \ge 0 - \rho_* = -\rho_*$.
    \\(b) If $ \varepsilon_j/\hat\varsigma \ge -x_* >  (\varepsilon_j - \mu)/\hat\varsigma$, the first $\rho$ is linear and the second is at most $\rho^*$. Then $r^\mu_j\ge \psi_*\{(\mu - \varepsilon_j)/\hat\varsigma - x_*\}  \ge 0$.
    \\(c) If $-x_* >  \varepsilon_j/\hat\varsigma > (\varepsilon_j - \mu)/\hat\varsigma$, then $\rho$ is linear in both arguments, so that $r^\mu_j=  \psi_* \mu/\hat\varsigma \ge 0$.   
    
    Now, insert the two bounds in (\ref{for-lemma-boundedness-sigmahat}). This gives a lower bound $\{R_n(\mu) - R_n(0)\}/h > L_n(\mu)$, where
    \begin{multline*}
        L_n(\mu) = - (\rho_* + \psi_* |\mu|/\hat\varsigma)(n-h)/h 
    \\  +\{\rho_* + \psi_*(|\mu|/\hat\varsigma - A_0/\hat\varsigma-x_*)\}(1-C_n/A_0^2) - \hat{\rho}_{\hat\varsigma}.
    \end{multline*}
    Rearrange this lower bound as
    $$
        L_n(\mu)=(2 - n/h - C_n/A_0^2)(\rho_* + \psi_*|\mu|/\hat\varsigma) 
        - \hat{\rho}_{\hat\varsigma} - (1 - C_n/A_0^2)\psi_*(x_* + A_0/\hat\varsigma) .
    $$
    
    By Assumption \ref{assumption-boundedness}, then $\forall\epsilon > 0$, $\exists C, n_0 > 0$, $\forall n > n_0$ we can find a set $\mathcal{S}_n$ with $\mathsf{P}(\mathcal{S}_n) > 1 - \epsilon$ so that on $\mathcal{S}_n$ we have $n/h<1/\lambda+\epsilon$, $C_n < C$ and $|\hat\varsigma-\varsigma|<\epsilon$ and in turn $|\hat{\rho}_{\hat\varsigma}-\tilde{\rho}_\varsigma|<\epsilon$. We choose $A_0$ so large that $C/A_0^2<\epsilon$.
    
    Consider the case $\psi_*=0$. Then $L_n(\mu)$ simplifies as
    \begin{align*}
        L_n(\mu) &=(2 - n/h - C_n/A_0^2) \rho_* - \hat{\rho}_{\hat\varsigma} \\
        &> (2-1/\lambda-2\epsilon) \rho_* -\tilde{\rho}_\varsigma-\epsilon
        = L.
    \end{align*}
    Since $\lambda>1/(2-\tilde{\rho}_\varsigma/\rho_*)$, 
    then $L>0$ for small $\epsilon$. 
    
    Consider the case $\psi_*>0$. Since $\lambda>1/2$, then $-1/\lambda>-2+3\epsilon$, for small $\epsilon$. Thus, on $\mathcal{S}_n$, we get $(2-n/h-C_n/A_0^2)>\epsilon>0$. We can then bound $L_n(\mu)$ from below using $|\mu| > \hat\varsigma B_0$ and get, on $\mathcal{S}_n$, that $L_n(\mu)>L$ where
    $$  L=\epsilon(\rho_* + \psi_*B_0) - [  \tilde{\rho}_{\varsigma} + \epsilon + \psi_*\{x_* + A_0/(\varsigma-\epsilon)\}]. $$
    Given $A_0$, we can now choose $B_0$ so large that $L>0$.

    For measurability, apply the argument of \cite{jennrich1969asymptotic}, see also \cite{johansen2019boundedness} and \cite{clarke2018robustness}, Chapter 4.
\end{proof}


\begin{proof}[Proof of Theorem \ref{theorem-consistency}]    
    The ob\-jec\-ti\-ve function (\ref{for:loss-function-sigma-hat}) is then $R_n(\mu) = \sum_{i=1}^n\rho\{ (\varepsilon_i-\mu)/\hat\varsigma\}.$
    Due to Theorem \ref{theorem-bounded} using Assumptions \ref{assumption-rho}, \ref{assumption-boundedness}, it suffices to consider $\mu$ in a compact set.

    Let $R_n^\alpha(\mu)=\sum_{i\in\zeta_n}\rho\{(\varepsilon_i-\mu)/\alpha\}$ for $\alpha=\varsigma,\hat\varsigma$. We show $R_n(\mu) = R_n^{\hat\varsigma}(\mu) + (n-h)\rho_*$ with large probability, uniformly in $\mu$. By the reverse triangle inequality $|\varepsilon_j-\hat\mu|/\hat\varsigma \ge (| \varepsilon_j|-|\hat\mu|)/\hat\varsigma$. Note, $\hat\varsigma\overset{p}{\to}\varsigma$ and $\min_{j\in\zeta_n^c} \varepsilon_j^2 \overset{p}{\to} \infty$, due to Assumptions \ref{assumption-boundedness}$(i)$, \ref{assumption-consistency}$(i)$. As $\mu$ is bounded, then $|\varepsilon_j-\hat\mu|/\hat\varsigma \ge x_*$ for large $n$. Thus, $\rho \{(\varepsilon_j-\mu)/\hat\varsigma\}=\rho_*$ and $R_n(\mu)$ has the desired form.
        
    We show $R_n^{\hat\varsigma}(\mu)$ and 
    $R_n^\varsigma(\mu)$ are close. Let $r_i(\mu)=\rho\{(\varepsilon_i-\mu)/\hat\varsigma\}-\rho\{(\varepsilon_i-\mu)/\varsigma\}$.  Since $\rho$ is Lipschitz by Assumption \ref{assumption-consistency}$(ii)$, then $|r_i(\mu)|\le C|\varepsilon_i-\mu||\hat\varsigma^{-1}-\varsigma^{-1}|$. Here, $|\varepsilon_i-\mu|\le |\varepsilon_i|+|\mu|$. 
    Since $\hat\varsigma\overset{p}{\to}\varsigma$ and $\mu$ is bounded, it suffices to argue that $h^{-1}\sum_{i\in\zeta_n} (|\varepsilon_i|+1)$ is bounded in probability. This follows from Assumption \ref{assumption-boundedness}$(ii)$.

    The $R_n^\varsigma(\mu)$ minimizers are near $\mu_\circ = 0$ by Assumption \ref{assumption-consistency}$(iii)$. Thus, $R_n^\varsigma(\mu)$ is bounded away from its minimum outside the vicinity of $\mu_\circ$. As $R_n^{\hat\varsigma}(\mu)$ 
    is close to $R_n^{\varsigma}(\mu)$, the same applies to $R_n^{\hat\varsigma}(\mu)$ and hence to $R_n(\mu)$.

    For measurability, apply the argument of \cite{jennrich1969asymptotic}.
\end{proof}


\begin{proof}[Proof of Theorem \ref{theorem-distribution}]
    Sums are taken over $i\in\zeta_n$.
    Let $u_i={\varepsilon_i}/{\varsigma}$, $v_i={\varepsilon_i-\hat\mu}/{\hat\varsigma}$, $w_i={(\varepsilon_i-\hat\mu_{\zeta_n})}/{\varsigma}$. 
    
    For $n$ large, the minimizers $\hat\mu$, $\hat\mu_{\zeta_n}$ solve  $0={\textstyle\sum}\dot\rho(v_i)$ and 
    $0={\textstyle\sum}\dot\rho(w_i)$.     
    The Mean Value Theorem gives
    \begin{align*}                  
        0&={\textstyle\sum}\dot\rho(u_i)+{\textstyle\sum}\ddot\rho(v_i^*)(v_i-u_i), \\
        0&={\textstyle\sum}\dot\rho(u_i)+{\textstyle\sum}\ddot\rho(w_i^*)(w_i-u_i).    
    \end{align*}
    for intermediate points $v_i^*$ and $w_i^*$. Subtract the equations while writing $v_i-u_i=(v_i-w_i)+(w_i-u_i)$ to get
    $$  0={\textstyle\sum}\ddot\rho(v_i^*)(v_i-w_i) +{\textstyle\sum}\big\{\ddot\rho(v_i^*)-\ddot\rho(w_i^*)\big\}(w_i-u_i).$$
    Replace differences $w_i-u_i=-{\hat\mu_{\zeta_n}}/{\varsigma}$ and
    \begin{align*}  
        v_i-u_i&=u_i({\varsigma}/{\hat\varsigma}-1)
         - {\hat\mu_{\zeta_n}}/{\hat\varsigma} -{(\hat\mu-\hat\mu_{\zeta_n})}/{\hat\varsigma}, 
    \end{align*}
    add \& subtract $\ddot\rho(u_i)$ to each $\ddot\rho$ and isolate $\hat\mu-\hat\mu_{\zeta_n}$ to get
    \begin{align*}
        &\quad (\hat\mu-\hat\mu_{\zeta_n}){\textstyle\sum}\big[\ddot\rho(u_i)+\big\{\ddot\rho(v_i^*)-\ddot\rho(u_i)\big\}\big] \\
        &= ({\varsigma}-{\hat\varsigma}) {\textstyle\sum}\big[\ddot\rho(u_i)+\big\{\ddot\rho(v_i^*)-\ddot\rho(u_i)\big\}\big] (u_i-{\hat\mu_{\zeta_n}}/{\varsigma}) \\
        &\quad - {\hat\mu_{\zeta_n}}(\hat\varsigma/{\varsigma}){\textstyle\sum}\big\{\ddot\rho(v_i^*) - \ddot\rho(w_i^*) \big\} .
    \end{align*}
    Using the triangle inequality, we can bound 
    \begin{align}
        &\quad {|\hat\mu-\hat\mu_{\zeta_n}|} \big\{ \big|{\textstyle\sum}\ddot\rho(u_i)\big| \label{asymptotic_inequality} 
         -{\textstyle\sum}\big|\ddot\rho(v_i^*)-\ddot\rho(u_i)\big|\big\} \\
        &\le |{\varsigma}-{\hat\varsigma}| \big|{\textstyle\sum}\ddot\rho(u_i)(u_i-{\hat\mu_{\zeta_n}}/{\varsigma})\big|\notag\\ 
        &\quad + |{\varsigma}-{\hat\varsigma}| {\textstyle\sum}\big|\ddot\rho(v_i^*)-\ddot\rho(u_i)\big| (|u_i|+{|\hat\mu_{\zeta_n}|}/{\varsigma}) \notag\\
        &\quad + {|\hat\mu_{\zeta_n}|}(\hat\varsigma/{\varsigma}){\textstyle\sum}\{\big|\ddot\rho(v_i^*) - \ddot\rho(w_i^*) \big| . \notag
    \end{align}
   
    By the Lipschitz property in Assumption \ref{assumption-distribution}$(ii)$ we get
    $$  \big|\ddot\rho(v_i^*) - \ddot\rho(u_i) \big| \le C |v_i^*-u_i| \le C |v_i-u_i| .    $$
    By the above expansion of $v_i-u_i$, we find, $\forall\epsilon>0$, $\exists$ a set $S_n$ with $P(S_n)>1-\epsilon$, on which
    $$
        {\textstyle\sum} \big|\ddot\rho(v_i^*) - \ddot\rho(u_i) \big| 
         \le  C {\textstyle\sum} \big( |u_i||{\varsigma}-{\hat\varsigma}| + |\hat\mu|\big)/\hat\varsigma \le \epsilon n,         
    $$
    as ${\textstyle\sum} |u_i| = O_p(h)$ by Assumption \ref{assumption-boundedness}$(ii)$, as $\hat\varsigma\overset{p}{\to}\varsigma$ by Assumption \ref{assumption-distribution}$(i)$, and given $\epsilon$ we choose $S_n$ and find a $C>0$ so that $|\hat\mu|<C n^{-1/2}$ on $S_n$ by Theorem \ref{theorem-consistency} using Assumptions \ref{assumption-rho}, \ref{assumption-boundedness}, \ref{assumption-consistency}. Similarly,
    $$  {\textstyle\sum} \big|\ddot\rho(v_i^*) - \ddot\rho(u_i) \big| |u_i| \le \epsilon n,\,  {\textstyle\sum} \big|\ddot\rho(w_i^*) - \ddot\rho(u_i) \big| \le \epsilon n. $$
    Finally, Assumption \ref{assumption-distribution}$(ii)$ implies
    $$  h^{-1} {\textstyle\sum}\ddot\rho(u_i) \overset{p}{\to} \mathsf{E}\ddot\rho(u_1)>0, \quad {\textstyle\sum} u_i\ddot\rho(u_i) = O_p(h^{1/2}).
    $$

    Return to (\ref{asymptotic_inequality}). Divide by $h$, insert the above bounds and use $\hat\varsigma-\varsigma=O_p(n^{-1/2})$ and $|\hat\mu_{\zeta_n}|\le C n^{-1/2}$ by Assumptions \ref{assumption-distribution}$(i,iv)$. Then, $\forall\epsilon>0$, $\exists$a large probability set $S_n$, on which
    $|\hat\mu-\hat\mu_{\zeta_n}| \mathsf{E}\ddot\rho(u_1) (1-\epsilon) \le \epsilon n^{-1/2} . $
    Thus, $\forall\epsilon>0$, $\exists S_n$, on which $|\hat\mu-\hat\mu_{\zeta_n}|\le \epsilon n^{-1/2}$. Use that $\epsilon$ is arbitrary.
    
    A measurable version of $\hat\mu$ is found as in \cite{jennrich1969asymptotic}. This has the assumed asymptotic distribution of $\hat\mu_{\zeta_n}$.
\end{proof}


\begin{proof}[Proof of Theorem \ref{theorem-non-consistency}]
    We have that $\hat\sigma = 1$ by Assumption \ref{assumption-non-consistency}$(i)$, so that
    $(y_i-\mu)/\hat\sigma=\varepsilon_i-\mu$.

    Since $\hat\mu$ is bounded in probability by Theorem \ref{theorem-bounded} using Assumptions \ref{assumption-rho}, \ref{assumption-boundedness}, it suffices to consider $|\mu|\le B$.
    Note that Assumption \ref{assumption-non-consistency}$(i)-(iii)$ implies Assumption \ref{assumption-boundedness}.

    `Outliers' satisfy $\varepsilon_j = \varepsilon_{(h)} + \xi$ with $\xi > 0$ for $j\not\in\zeta_n$, while $\varepsilon_{(h)}=\max_{i\in\zeta_n}\varepsilon_i \to\infty$ in probability by Assumption \ref{assumption-non-consistency}$(ii,v)$. As $|\varepsilon_j-\mu| \ge | \varepsilon_j|-|\mu|$ by the reverse triangle inequality, $\varepsilon_j\to\infty$ and $|\mu|\le B$, then $|\varepsilon_j-\mu|\to\infty$.
    Thus, $\rho (\varepsilon_j-\mu) = \rho_* + \psi_*\{\varepsilon_{(h)} + \xi - \mu - x_*\}$
    due to Assumption \ref{assumption-rho} to $\rho$.  
    We get that on a large probability set, then
    $R_n(\mu)= \tilde{R}_n(\mu) + \tilde{C}_n$, where 
    \begin{align*}
        \tilde{R}_n(\mu)
        &= {\textstyle\sum}_{i\in\zeta_n} \left\{\rho(\varepsilon_i-\mu)\right\} - (n-h) \psi_*\mu, \\
        \tilde{C}_n 
        &= (n-h) \psi_*\{\varepsilon_{(h)} + \xi\} + (n-h)(\rho_* - \psi_*x_*).
    \end{align*}
    Note, $\tilde{C}_n$ is constant in $\mu$. Thus, we only analyze $\tilde{R}_n(\mu)$.

    Since $h^{-1} \sum_{i\in\zeta_n}\rho( \varepsilon_i-\mu) \to \mathsf{E}\rho( \varepsilon_1 - \mu)$ in probability uniformly 
    for $\mu$ near 0
    by Assumption \ref{assumption-non-consistency}$(iii)$, we get $$h^{-1} \tilde{R}_n(\mu) \to \mathsf{E}\{\rho(\varepsilon_1 - \mu)\} - (1/\lambda-1)\psi_*\mu.$$
    In this limit, the
    first term has a zero derivative in $\mu$ at $0$ 
    by Assumption \ref{assumption-non-consistency}$(iv)$. The second, subtracted, term 
    is positive since, by assumption, $\psi_* > 0$, $\lambda < 1$. Thus, the derivative at $0$ is non-zero so $\hat\mu$ is inconsistent. 
\end{proof}

\begin{proof}[Proof of Theorem \ref{theorem-iqr}]   
    Let $\hat{v}_{p}$, $\hat{v}_p^\circ$ denote the $p$-quantiles of $\varepsilon_i$ for $i=1,\dots ,n$ and for $i\in\zeta_n$, respectively. Similarly, let $\hat{v}_{p,abs}$, $\hat{v}_{p,abs}^\circ$ be $p$-quantiles of $|\varepsilon_i|$.

    $(a)$ 
    We get worst bounds for
    $\hat{v}_p$ 
    by placing all `outliers' on one side of the `good' observations:
    $$
         \hat{v}^\circ_{ 1-3/(4\lambda) } \le \hat{v}_{1/4} \le \hat{v}^\circ_{ 1/(4\lambda) }, \quad \hat{v}^\circ_{ 1-1/(4\lambda) } \le \hat{v}_{3/4} \le \hat{v}^\circ_{ 3/(4\lambda) } .
    $$
    The `good' quantiles are bounded in probability by Assumption \ref{assumption-iqr-bounded} since $0<1-3/(4\lambda)\le 1/(4\lambda)<1$ and $0<1-1/(4\lambda)\le 3/(4\lambda)<1$ when $\lambda>3/4$ . 
    Hence, $\hat{\sigma}_{\textsc{IQR}}$ is bounded in probability. 

    $(b)$ 
    Combine with the triangle inequality to get
    $$
        |y_i-\hat{q}_{1/2}| =  |\varepsilon_i-\hat{v}_{1/2}| \le  |\varepsilon_i| + |\hat{v}_{1/2}|.
    $$
    Taking the median, we get that $\hat{\sigma}_{\textsc{MAD}}=\mathrm{med}\, |y_i-\hat{q}_{1/2}|$ is bounded by $\hat{v}_{1/2,abs} + |\hat{v}_{1/2}|$. Argue as in $(a)$ that $\hat{v}^\circ_{ 1-1/(2\lambda) } \le \hat{v}_{1/2} \le \hat{v}^\circ_{ 1/(2\lambda) }$ and $ \hat{v}^\circ_{ 1-1/(2\lambda),abs } \le \hat{v}_{1/2,abs} \le \hat{v}^\circ_{ 1/(2\lambda),abs } .
    $
    The `good' quantiles are bounded in probability when $\lambda>1/2$. Hence $\hat{\sigma}_{\textsc{MAD}}$ is bounded. 
\end{proof}

\begin{proof}[Proof of Theorem \ref{prop-inconsistent-scale}]
    (a) By Assumption \ref{assumption-igr-inconsistency}, $\hat{q}_{p} \overset{p}{\rightarrow} q^{\mathsf{F}}_{(p - \varrho)/\lambda}$ for $p=1/4$, $3/4$.
    Thus, $\hat\sigma_{IQR}\overset{p}{\rightarrow} \varsigma_{IQR}$. 

    (b). Again, $\hat{q}_{1/2} \overset{p}{\rightarrow} q^{\mathsf{F}}_{(1/2 - \varrho)/\lambda}$.
    The `good' errors 
    are continuously distributed
    by Assumption \ref{assumption-igr-inconsistency}$(i)$. 
    Thus, we analyze $\mathrm{med}\, (|\varepsilon_i - c |)$, where $\mathsf{F}(c)=(1/2-\varrho)/\lambda$.
    
    We 
    find a $d$ such that $\# \left(i \colon |\varepsilon_i - c| \leq d \right) = n/2$.
    The `outlier' errors 
    drift to infinity by Assumption \ref{assumption-igr-inconsistency}$(ii)$, so as $n \rightarrow \infty$ we have that, for any finite valued $d$,
    \begin{align*}
        \mathsf{P}\left\{\left(i \colon |\varepsilon_i - c| \le d \right) = \left(i \in \zeta_n \colon  |\varepsilon_i - c| \le d   \right) \right\} \rightarrow 1.
    \end{align*}
    For any finite valued $d$, as $n \rightarrow \infty$,
    $$
        h^{-1} \#\left(i \in \zeta_n \colon  |\varepsilon_i - c| \le d   \right) 
        \overset{p}{\to} \mathsf{F}(c+d) - \mathsf{F}(c-d) .
    $$
    Solutions $d$ to
    $\# (i \in \zeta_n \colon  |\varepsilon_i - c| \le d   )=h/2$
    converge to solutions of $\mathsf{F}(c+d) - \mathsf{F}(c-d)=1/2$ as $\mathsf{F}$ is continuous.
    Then, solutions to $\# (i: | \varepsilon_i - \hat{q}_{1/2} |\le d) = n/2$ converge
    to solutions of $\lambda\{\mathsf{F}(c+d)-\mathsf{F}(c-d)\}=1/2$. 
\end{proof}

\begin{remark}\label{rem:consistency_IQR}
    We argue that $\varsigma_{\textsc{IQR}}>1$ for normal `good' errors and $3/4<\lambda<1$. Write $\varsigma_{\textsc{IQR}}=s_{\lambda,\varrho}/s_{1,0}$, where $s_{\lambda, \varrho}=q^\Phi_{(3/4-\varrho)/\lambda}-q^\Phi_{(1/4-\varrho)/\lambda}$. As $s_{\lambda,\varrho}=s_{\lambda,1-\lambda-\varrho}$, it suffices to analyze $\varrho \le (1-\lambda)/2$.

    If $\varrho=(1-\lambda)/2$, then $s_{\lambda,\varrho}=q^\Phi_{1/2+1/(4\lambda)}-q^\Phi_{1/2-1/(4\lambda)}$, which decreases to $s_{1,0}$ as $\lambda\uparrow 1$. Thus, $\varsigma_{\textsc{IQR}}>1$.

    If $\varrho=(1-\lambda)/2-\alpha\lambda$ for $0<\alpha<(1-\lambda)/(2\lambda)$ then $s_{\lambda,\varrho}=q^\Phi_{1/2+1/(4\lambda)+\alpha}-q^\Phi_{1/2-1/(4\lambda)+\alpha}$. This increases in $\alpha$, since the argument of the first term term is furthest from $1/2$ for any $\alpha$. Thus, $\varsigma_{\textsc{IQR}}>1$ by the previous result.
\end{remark}
            
\begin{remark}
    \label{rem:consistency_MAD}
    We argue that $\varsigma_{\textsc{MAD}}>1$ for normal `good' errors, $1/2<\lambda<1$ and any $\varrho$.
    Given $c$, we choose $d$ so that $A_{c,d}=\Phi(c+d)-\Phi(c-d)=1/(2\lambda)$. We show $d>q^\Phi_{3/4}$. 
    Now, $A_{c,d}\le\Phi(d)-\Phi(-d)$, which  increases in $d$. Thus, $d$ is least for $c=0$. As $\lambda<1$ then $1/(2\lambda)>1/2$, and $d>q^\Phi_{3/4}$, which solves $\Phi(d)-\Phi(d)=1/2$.
%
\end{remark}


\setlength{\bibsep}{0.0pt}
\bibliographystyle{apalike}
\bibliography{bent_journals_long,bibliography}

\end{document}